\documentclass[11pt]{article}

\usepackage{epsfig}
\usepackage{graphicx}
\usepackage{amsbsy}
\usepackage{amsmath}
\usepackage{amsfonts}
\usepackage{amssymb}
\usepackage{textcomp}
\usepackage{hyperref}
\usepackage{aliascnt}

\newcommand{\mcm}[3]{\newcommand{#1}[#2]{{\ensuremath{#3}}}} 

\mcm{\tuple}{1}{\langle #1 \rangle}
\mcm{\name}{1}{\ulcorner #1 \urcorner}
\mcm{\Nbb}{0}{\mathbb{N}}
\mcm{\Zbb}{0}{\mathbb{Z}}
\mcm{\Rbb}{0}{\mathbb{R}}
\mcm{\Cbb}{0}{\mathbb{C}}
\mcm{\Qbb}{0}{\mathbb{Q}}
\mcm{\Acal}{0}{\cal A}
\mcm{\Bcal}{0}{\cal B}
\mcm{\Ccal}{0}{\cal C}
\mcm{\Dcal}{0}{\cal D}
\mcm{\Ecal}{0}{\cal E}
\mcm{\Fcal}{0}{\cal F}
\mcm{\Gcal}{0}{\cal G}
\mcm{\Hcal}{0}{\cal H}
\mcm{\Ical}{0}{\cal I}
\mcm{\Jcal}{0}{\cal J}
\mcm{\Kcal}{0}{\cal K}
\mcm{\Lcal}{0}{\cal L}
\mcm{\Mcal}{0}{\cal M}
\mcm{\Ncal}{0}{\cal N}
\mcm{\Ocal}{0}{{\cal O}}
\mcm{\Pcal}{0}{{\cal P}}
\mcm{\Qcal}{0}{{\cal Q}}
\mcm{\Rcal}{0}{{\cal R}}
\mcm{\Scal}{0}{{\cal S}}
\mcm{\Tcal}{0}{{\cal T}}
\mcm{\Ucal}{0}{{\cal U}}
\mcm{\Vcal}{0}{{\cal V}}
\mcm{\Wcal}{0}{{\cal W}}
\mcm{\Xcal}{0}{{\cal X}}
\mcm{\Ycal}{0}{{\cal Y}}
\mcm{\Zcal}{0}{{\cal Z}}
\mcm{\Mfrak}{0}{\mathfrak M}

\mcm{\restric}{0}{\upharpoonright}
\mcm{\upset}{0}{\uparrow}
\mcm{\onto}{0}{\twoheadrightarrow}
\mcm{\smallNbb}{0}{{\small \mathbb{N}}}
\DeclareMathOperator{\preop}{op}
\mcm{\op}{0}{^{\preop}}

{\begin{array}{c}
\setlength{\unitlength}{1em}}%
{\end{array}}

\usepackage{amsthm}

\newcommand{\theoremize}[2]{\newaliascnt{#1}{thm} \newtheorem{#1}[#1]{#2} \aliascntresetthe{#1}}

\theoremstyle{plain}

\theoremize{lem}{Lemma}
\theoremize{skolem}{Skolem}
\theoremize{fact}{Fact}
\theoremize{sublem}{Sublemma}
\theoremize{claim}{Claim}
\theoremize{obs}{Observation}
\theoremize{prop}{Proposition}
\theoremize{cor}{Corollary}
\theoremize{que}{Question}
\theoremize{oque}{Open Question}
\theoremize{con}{Conjecture}

\theoremstyle{definition}
\theoremize{dfn}{Definition}
\theoremize{rem}{Remark}
\theoremize{eg}{Example}
\theoremize{exercise}{Exercise}
\theoremstyle{plain}

\usepackage{verbatim}
\usepackage{enumerate}
\usepackage[all]{xy}

\usepackage{subfig}

\usepackage{subfig}
\usepackage{wrapfig}
\usepackage{caption}
\usepackage[T1]{fontenc}

\title{Stotting in positional games\footnote{This research was partly supported by the ANR project P-GASE (ANR-21-CE48-
0001-01).}}
\author{Johannes Carmesin\thanks{TU Freiberg, funded by DFG, project number 546892829.} \and Yannick Mogge\thanks{Université Claude Bernard Lyon 1}}
\date{}

\newcommand{\sm}{\setminus}

\mcm{\Fbb}{0}{\mathbb{F}}

\begin{document}

\makeatletter
\renewcommand{\@fnsymbol}[1]{%
  \ifcase#1\or *\or **\or ***\else\arabic{#1}\fi}
\makeatother

\maketitle

\begin{abstract}
We introduce variants of the Maker–Breaker and Waiter–Client games, which we call \emph{stotting}, in which a player grants a slight advantage to the opponent. We prove that a winning strategy in either stotting variant yields winning strategies for both Maker and Waiter in the classical setting.
Several existing Maker strategies in the literature in fact win with stotting, and therefore automatically provide both classical winning strategies (and similarly for stotting Waiter).

Knox previously disproved a conjecture of Beck asserting that whenever Maker wins the Maker–Breaker game, Waiter also wins the corresponding Waiter–Client game; in this sense, our framework may be viewed as a way of repairing Beck’s conjecture.
\end{abstract}

\section{Introduction}

In this short note, we propose a way to unify winning strategies for Maker and Waiter in the well-studied Maker–Breaker and Waiter–Client games.

Here we consider the standard \emph{Maker-Breaker game} played on a hypergraph $\Hcal=(X,\mathcal{F})$ with Breaker being the first player to make a move. The vertex set $X$ of $\Hcal$ 
is referred to as the \emph{board} and the edge set $\Fcal \subset 2^X$ represents the \emph{winning sets} for Maker. In the \emph{Waiter-Client game} the elements of the board are distributed differently: in each turn, Waiter offers two elements, Client takes one, and Waiter receives the other\footnote{In case that at the end of the game there is only one element of the board available, this element goes to Client.}. Here, $\Fcal$ encodes the winning sets for Waiter.

In games such as Lehman’s connectivity game~(\cite{CMP2009},\cite{L1964}), the Hamiltonicity game~(\cite{CE1978},\cite{CGHHMM2020}), fixed spanning tree game~(\cite{CFGHL2013},\cite{CGHHMM2020}), and tree universality game~\cite{AABCHM2025}, 
Maker and Waiter are known to have winning strategies of a similar flavour. This led Beck~\cite{B2002} to conjecture that this phenomenon holds in general, and later Csernenszky, Mándity and Pluhár~\cite{CMP2009} stated the conjecture, that whenever Maker has a winning strategy, so does Waiter in the corresponding game on the same hypergraph (see Conjecture 1 in~\cite{CMP2009}). However, this conjecture was later disproved by Knox~\cite{K2012}. 
In this note, we propose two games that may be regarded as a way of repairing Beck’s conjecture, as follows. 
We shall prove that if Maker has a winning strategy with \lq stotting\rq—meaning that they grant a slight advantage to their opponent—then they also have a winning strategy as Waiter; see \autoref{stotting}. We also give a corresponding construction with the roles of \lq Maker\rq\ and \lq Waiter\rq\ exchanged. 
We illustrate the unifying character of our approach by observing that several existing strategies in the literature are in fact stotting, and therefore simultaneously yield winning strategies for both Maker and Waiter in the classical setting. In this sense, our framework may be viewed as restoring Beck’s conjecture.

\begin{figure}[t]
\vspace{-0.8cm}
\centering
\begin{minipage}{0.13\textwidth}
  \centering
  \includegraphics[width=\linewidth]{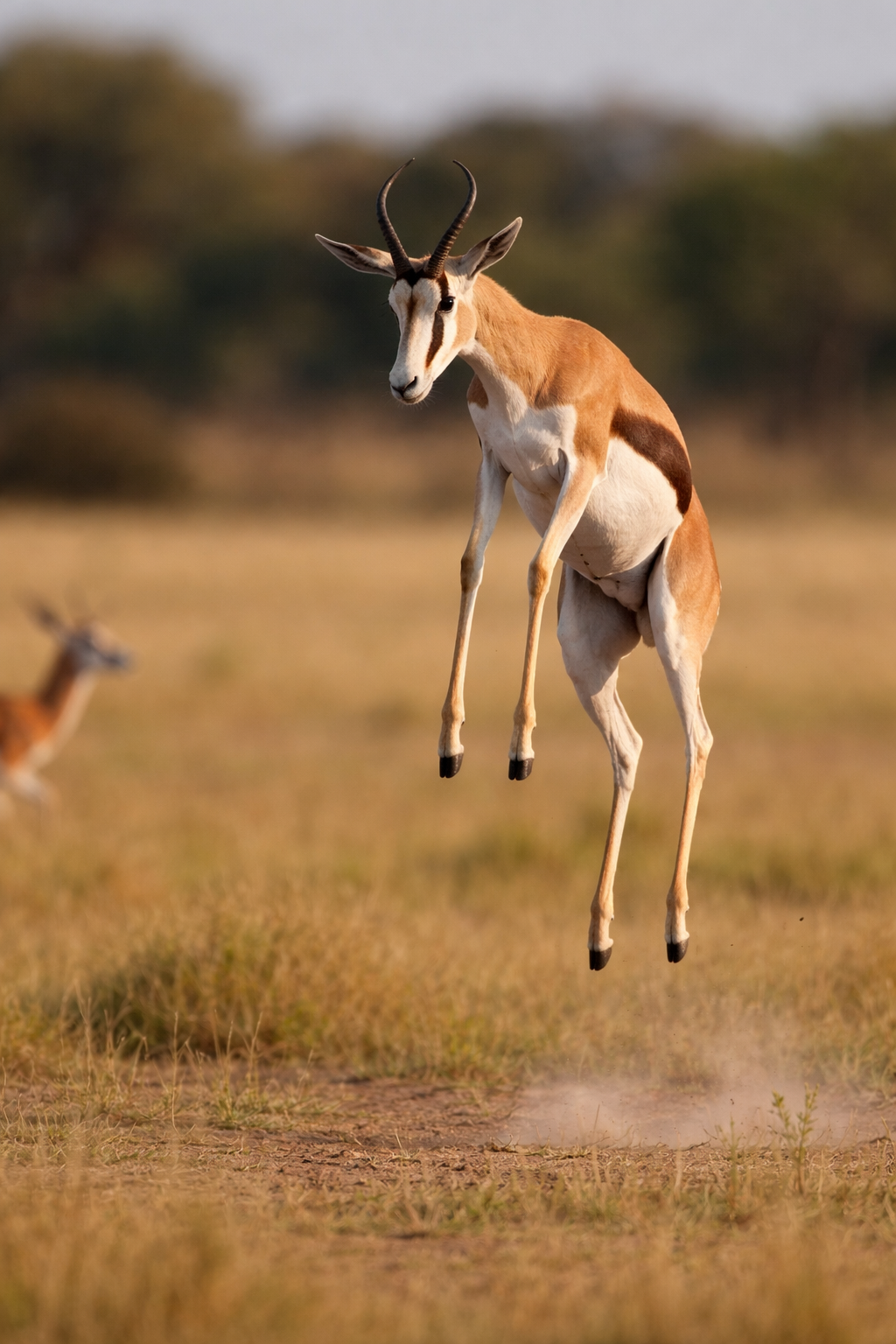}
\end{minipage}\hfill
\begin{minipage}{0.83\textwidth}
\vspace{0.4cm}
\captionsetup{font=footnotesize}
  \caption{\footnotesize When a lion pursues a herd of antelopes, some individuals engage in \emph{stotting}—high, stiff-legged jumps that function as an honest signal of escape ability, advertising to the predator that they are in good condition and unlikely to be worth chasing. In this paper, we use the term \lq stotting\rq\ to describe games where players can win so strongly that they can even win another game.}
  \label{stotting}
\end{minipage}
\vspace{-0.8cm}
\end{figure}

The \emph{stotting Maker-Breaker game} is defined like the \lq Maker-Breaker game\rq\ except that before each turn (which consists of two moves) Breaker decides who goes first.

\begin{prop}\label{p1}
If stotting Maker has a winning strategy on $\Hcal$, then they also have a winning strategy as Waiter on $\Hcal$.
\end{prop}

In \autoref{p2} below we observe that a sufficient criterion for stotting Maker to have a winning strategy is that they win the Maker-Breaker game with Breaker always taking two moves instead of one.

\begin{eg}\label{Ham_eg}
Given $n,b\in \Nbb$, the (1:b) \emph{Hamiltonicity game on $K_n$} is the game for the hypergraph $\Hcal_n$ whose board is the set of edges of $K_n$ and whose winning sets are the Hamilton cycles of $K_n$, and during each turn Breaker is allowed to claim up to $b$ edges. It is well-known (see~\cite{K2011}), that Maker has a winning strategy for the (1:b) Maker-Breaker Hamiltonicity game on $K_n$ if $b=(1+o(1))\frac{n}{\ln n}$ and thus on the (1:2) Maker-Breaker Hamiltonicity game on $K_n$ with $n$ sufficiently large. 
By \autoref{p2} and \autoref{p1} combined, this fact implies that Waiter has a winning strategy on $\Hcal_n$.
\end{eg}

The \emph{stotting Waiter-Client} game is defined like the classical \lq Waiter-Client game\rq\ except that 
in each turn, Client selects an unclaimed element of the board, then Waiter adds a second one, and offers both to Client, who picks one and gives the other one to Waiter. 

\begin{prop}\label{p3}
If stotting Waiter has a winning strategy on $\Hcal$, then they have a winning strategy as Maker on $\Hcal$.
\end{prop}

In \autoref{lehman_prop} below we shall construct a winning strategy for stotting Waiter in the classic Lehman connectivity game (see~\cite{L1964}), which by \autoref{p3} entails two winning strategies for Waiter and Maker, respectively, unifying existing proofs in the literature. 
The main body of the paper follows in the next section, and then we finish with a short outlook.

\section{Proofs}

In this short section, we prove the propositions stated or mentioned in the introduction. 
Given two vertices $v_1$ and $v_2$ of a hypergraph $\Hcal=(X,\Fcal)$, by $\Hcal\sm v_1/v_2$ we denote the hypergraph with vertex set $X-v_1-v_2$ and the edge set consists of all $f$ so that $f$ or $f+v_2$ are in $\Fcal$ and $v_1\notin f$.

\begin{proof}[Proof of \autoref{p1}.]
We are going to show that if stotting Maker has a winning strategy on $\Hcal$, then they also have a winning strategy as Waiter on $\Hcal$.
Arguing by induction, it suffices to find two vertices $v_1$ and $v_2$ of $\Hcal$ so that Maker has a winning strategy for the stotting Maker-Breaker games on $\Hcal\sm v_1/v_2$ and  $\Hcal\sm v_2/v_1$. Consider the stotting Maker-Breaker game on $\Hcal$ with Breaker deciding that Maker takes the first move. Maker's winning strategy determines which vertex to claim in the next move; call it $v_1$. Hence for every vertex $w$, Maker has a winning strategy for the  stotting Maker-Breaker game on $\Hcal\sm w/v_1$. 

Now consider the stotting Maker-Breaker game in $\Hcal$ with Breaker deciding that they take the first move, and playing $v_1$. Denote Maker's response according to their winning strategy by $v_2$.
Then Maker has a winning strategy for the  stotting Maker-Breaker game on $\Hcal\sm v_1/v_2$. This completes the proof. 
\end{proof}


\begin{prop}\label{p2}
If Maker has a winning strategy in a $(1:2)$ Maker-Breaker game on $\Hcal = (X,\Fcal)$, then they have a winning strategy as stotting Maker 
on $\Hcal$.
\end{prop}

\begin{proof}
If Maker has a winning strategy in the $(1:2)$ Maker-Breaker game, then
also have a winning strategy in the lazy variant, where Breaker is allowed to skip turns (indeed, not playing a move is always worse than making a rubbish move); hence below we may assume that Maker still has a winning strategy if from time to time Breaker decides to skip moves. 
In the $(1:2)$ Maker-Breaker game, a turn consists of Breaker taking two moves and then Maker taking a single move. 
Now we group moves together differently into turns by allocating the first move of each turn to the previous turn, and skipping the very first move of Breaker. Now each turn starts with a move of Breaker, then a move of Maker and then a move of Breaker. If Breaker restricts themself to skip exactly one move each turn, we obtain the stotting Maker-Breaker game. Hence if Maker has a winning strategy in the $(1:2)$ Maker-Breaker game, then they have a winning strategy as stotting Maker.
\end{proof}

\begin{proof}[Proof of \autoref{p3}.]
Arguing by induction, we are going to show that if stotting Waiter has a winning strategy on $\Hcal$, then they have a winning strategy as Maker on $\Hcal$. Since stotting Waiter has a winning strategy on $\Hcal$ by assumption, for every $v \in X$ we can find an element $w \in X$ such that if Client suggests $v$ as the first element, Waiter can then offer $w$ as the second element, therefore winning the stotting Waiter-Client game on $\Hcal\setminus v / w$ and $\Hcal \setminus w / v$. Thus, in the Maker-Breaker game on $\Hcal$, Maker can reply to Breaker claiming $v$ by claiming $w$, thus reducing the game to the Maker-Breaker game on $\Hcal\setminus v /w$, which turns out to be a win for Maker, since we know that stotting Waiter wins on $\Hcal\setminus v /w$.
\end{proof}


A classical result in positional games is \emph{Lehman's connectivity game} which is played on the ground set of a matroid $M$ that has two disjoint spanning sets, and the winning sets are the spanning sets. 
The following unifies the known proofs that Maker and Waiter can win the connectivity game.

\begin{prop}\label{lehman_prop}
Stotting Waiter wins Lehman's connectivity game.
\end{prop}

In order to prove \autoref{lehman_prop} we shall use the following.

\begin{lem}[Symmetric base exchange \cite{brualdi1969comments}]\label{matroid-lem}
For any matroid $M$ with two bases $B_1$ and $B_2$ and any $e\in B_1$ there is $f\in B_2$ such that
the sets $B_1-e+f$ and $B_2-f+e$ are bases. \qed
\end{lem}

\begin{proof}[Proof of \autoref{lehman_prop}.]
Let $M$ be a matroid with two disjoint spanning sets, which contain bases $B_1$ and $B_2$. Our aim is to construct a winning strategy for stotting Waiter in Lehman's connectivity game on the matroid $M$. 
We may assume, and we do assume, that every element of the ground set of $M$ is in one of these bases; hence the bases $B_1$ and $B_2$ partition the ground set $E(M)$. 

Arguing by induction, it suffices to find for every element $e$ of $M$ an element $f$ of $M$ so that the sets $B_1-e-f$ and $B_2-e-f$ are bases in the two matroids $M\sm e/f$ (the matroid obtained from $M$ by deleting $e$ and contracting $f$) and $M\sm f/e$. 
\autoref{matroid-lem} yields for every element $e$ a suitable element $f$, completing the proof. 
\end{proof}

\section{Outlook}

The stotting variants give one possible repair of Beck's conjecture: a winning
strategy for stotting Maker yields winning strategies for both Maker and
Waiter in the classical games, and the analogous statement holds for stotting
Waiter. This opens the way to pursue Beck's original programme in the
strengthened setting of stotting strategies. In many natural positional games,
Maker and Waiter are known to win by similar-looking strategies; revisiting
these examples from the stotting viewpoint may reveal common robust strategies
that unify the Maker and Waiter proofs.

Since for every Maker--Breaker game there is a transversal game, that is, a
Maker--Breaker game in which Maker wins if and only if Breaker wins in the
original game, the symmetry between Maker and Breaker also yields a natural
notion of stotting Breaker. Thus stotting variants arise naturally for Maker,
Waiter and Breaker. It remains open whether there is an equally natural notion
of stotting Client, and whether such a notion would be useful in examples.

This note does not attempt to classify when a classical winning strategy can
be strengthened to a stotting strategy. Natural next questions are how the
usual threshold biases compare with the corresponding stotting thresholds,
and whether there are natural Maker-winning games, beyond artificial
counterexamples, in which Maker cannot win with stotting.

\bibliographystyle{plain}
\bibliography{literatur}

\end{document}